\documentclass[11pt]{article}

\usepackage[margin=1in]{geometry}
\usepackage[T1]{fontenc}
\usepackage{lmodern}
\usepackage{microtype}
\usepackage{amsmath,amssymb,amsthm,mathtools}
\usepackage{enumitem}
\usepackage{hyperref}
\usepackage[nameinlink,noabbrev]{cleveref}
\usepackage{xcolor}

\hypersetup{
  colorlinks=true,
  linkcolor=blue!45!black,
  citecolor=blue!45!black,
  urlcolor=blue!45!black,
  pdftitle={Linear Turan Numbers of Four-Edge Uniform Paths via Incidence Rank},
  pdfauthor={Mahesh Ramani}
}

\setlist{nosep,leftmargin=1.6em}
\newtheorem{theorem}{Theorem}[section]
\newtheorem{lemma}[theorem]{Lemma}
\newtheorem{corollary}[theorem]{Corollary}

\theoremstyle{definition}

\theoremstyle{remark}

\DeclareMathOperator{\rank}{rank}
\DeclareMathOperator{\exlin}{ex^{\mathrm{lin}}}
\newcommand{\R}{\mathbb R}
\newcommand{\cF}{\mathcal F}
\newcommand{\cA}{\mathcal A}
\newcommand{\cS}{\mathcal S}
\newcommand{\one}{\mathbf 1}

\title{\bfseries Linear Tur\'an Numbers of Four-Edge Uniform Paths\\via Incidence Rank}
\author{Mahesh Ramani}
\date{\today}

\begin{document}
\maketitle

\begin{abstract}
Let $P_4^r$ denote the $r$-uniform expansion of the graph path with four edges.  A conjecture of Adak and Verma asserts that every $n$-vertex linear $r$-uniform $P_4^r$-free hypergraph has at most $(r+1)n/r$ edges, with equality precisely for vertex-disjoint unions of Steiner systems $S(2,r,r^2)$.  The conjecture is proved for every $r\ge2$.

The main ingredient is an incidence-rank inequality.  If $N(H)$ is the edge--vertex incidence matrix of a linear $r$-uniform hypergraph whose line graph is a cograph, then
\[
 (r+1)\rank_{\R}N(H)\ge r|E(H)|.
\]
Equality holds exactly when every edge-containing component is an $S(2,r,r^2)$.  The proof follows the join decomposition of a connected cograph.  At each join node, the row-difference spaces of the co-components are mutually orthogonal, and the possible rank defect is determined by a distinguished collection of balanced co-components.  Perron--Frobenius theory identifies the smallest balanced pieces as parallel classes of $r$ disjoint $r$-sets, while an orthogonality argument bounds their number by $r+1$.  The equality case then reconstructs the Steiner system.  Since $\rank N(H)\le |V(H)|$, the linear Tur\'an bound and its equality characterization follow.
\end{abstract}

\section{Introduction}

A hypergraph is \emph{linear} if two distinct hyperedges meet in at most one vertex.  For a linear $r$-uniform hypergraph $F$, let $\exlin_r(n,F)$ denote the maximum number of edges in an $n$-vertex linear $r$-uniform hypergraph containing no copy of $F$.  The terminology and systematic study of linear Tur\'an numbers were developed for linear cycles by Collier-Cartaino, Graber, and Jiang \cite{CollierCartainoGraberJiang2018}.  For $r=2$, the bound below is the classical Erd\H{o}s--Gallai bound for a path on five vertices \cite{ErdosGallai1959}.  Gy\'arf\'as, Ruszink\'o, and S\'ark\"ozy initiated the corresponding study of acyclic linear triple systems and determined, among other small cases, the sharp bound for the four-edge $3$-uniform path \cite{GyarfasRuszinkoSarkozy2022}.  Zhang and Wang subsequently considered acyclic linear $4$-graphs and stated the analogous result for the four-edge $4$-uniform path \cite{ZhangWang2026}.  Adak and Verma identified a gap in that argument, supplied a new proof for $r=4$, and formulated the general conjecture treated below; they also proved several partial results under degree hypotheses \cite{AdakVerma2026}.  Thus the previously unresolved range is $r\ge5$, while the argument below applies uniformly to every $r\ge2$.

For a graph $G$, its $r$-uniform expansion is obtained by replacing every graph edge by an $r$-set, with the new vertices used for distinct graph edges pairwise disjoint.  Let $P_4^r$ be the expansion of the graph path with four edges.  Thus a copy of $P_4^r$ consists of four hyperedges $e_1,e_2,e_3,e_4$ such that consecutive pairs intersect, nonconsecutive pairs are disjoint, and the three consecutive intersection vertices are distinct.

The main extremal result settles the conjecture of Adak and Verma.

\begin{theorem}\label{thm:main}
For every $r\ge2$ and every linear $r$-uniform $P_4^r$-free hypergraph $H$,
\begin{equation}\label{eq:main-bound}
 |E(H)|\le \frac{r+1}{r}|V(H)|.
\end{equation}
Equality holds if and only if the edge-containing components of $H$ are Steiner systems $S(2,r,r^2)$ and $H$ has no isolated vertices.  Equivalently, the equality examples are vertex-disjoint unions of copies of $S(2,r,r^2)$.
\end{theorem}

An $S(2,r,r^2)$ has $r^2$ vertices, $r(r+1)$ blocks, and degree $r+1$, so its edge density is $(r+1)/r$.  It is $P_4^r$-free.  Indeed, fix a block $B$ and a point $x\notin B$.  Of the $r+1$ blocks through $x$, exactly $r$ meet $B$, one through each point of $B$, and the remaining block is the unique block through $x$ disjoint from $B$.  In a putative path $e_1,e_2,e_3,e_4$, the point $e_3\cap e_4$ lies outside $e_1$, whereas both $e_3$ and $e_4$ would be distinct blocks through that point and disjoint from $e_1$, a contradiction.  The equality statement in \cref{thm:main} is conditional on the existence of $S(2,r,r^2)$; in particular, it does not assert that equality is attainable for every $r$ or every $n$.

The proof of \cref{thm:main} is based on a stronger rank statement.  For a hypergraph $H$, let $N(H)$ be its $0$--$1$ edge--vertex incidence matrix, with rows indexed by $E(H)$ and columns indexed by $V(H)$.

\begin{theorem}[Incidence-rank theorem]\label{thm:rank}
Let $H$ be a linear $r$-uniform hypergraph, where $r\ge2$, and suppose that its line graph is a cograph.  If $m=|E(H)|$, then
\begin{equation}\label{eq:rank-bound}
 (r+1)\rank_{\R}N(H)\ge rm.
\end{equation}
Equality holds if and only if each edge-containing component of $H$ is a Steiner system $S(2,r,r^2)$.
\end{theorem}

The identity
\[
 N(H)N(H)^{\mathsf T}=A(L(H))+rI
\]
shows that \cref{thm:rank} may equivalently be viewed as a sharp upper bound on the multiplicity of the adjacency eigenvalue $-r$ in cographs that arise as line graphs of linear $r$-uniform hypergraphs.  Rank and eigenvalue multiplicity questions for general cographs have been studied extensively; see, for example, Royle \cite{Royle2003}, Jacobs, Trevisan, and Tura \cite{JacobsTrevisanTura2018}, and Ghorbani \cite{Ghorbani2019}.  The argument here uses the additional incidence representation and linearity of the underlying hypergraph.

A linear hypergraph is $P_4^r$-free exactly when its line graph is a cograph.  The standard union--join decomposition of cographs, due to Seinsche \cite{Seinsche1974}, therefore provides the recursive structure for the proof.  At a join node, the row-difference spaces of the co-components are orthogonal.  The exact rank loss is measured by balanced co-components, whose disjointness graphs are controlled by Perron--Frobenius theory.  Minimal balanced co-components are parallel classes, and a dimension count permits at most $r+1$ of them.  This both proves the inequality and determines the equality case.  Finally, \cref{thm:main} follows from $\rank N(H)\le |V(H)|$.

\section{Cographs and incidence vectors}

The \emph{line graph} $L(H)$ has vertex set $E(H)$, with two vertices adjacent when the corresponding hyperedges intersect.  A graph is a \emph{cograph} if it has no induced path on four vertices.  The standard decomposition theorem of Seinsche \cite{Seinsche1974} states that every cograph with at least two vertices is disconnected or has disconnected complement.  In particular, if a cograph is connected and nontrivial, the connected components of its complement form a nontrivial join decomposition.

\begin{lemma}\label{lem:line-cograph}
A linear $r$-uniform hypergraph $H$ is $P_4^r$-free if and only if $L(H)$ is a cograph.
\end{lemma}

\begin{proof}
A copy of $P_4^r$ gives an induced $P_4$ in the line graph.  Conversely, suppose four hyperedges induce a graph path in $L(H)$.  Consecutive hyperedges intersect and nonconsecutive hyperedges are disjoint.  Linearity makes each consecutive intersection a single vertex.  Let $v_i=e_i\cap e_{i+1}$ for $i=1,2,3$.  No two of $v_1,v_2,v_3$ can coincide, since any such coincidence would make a nonconsecutive pair of hyperedges intersect.  Hence the four hyperedges form $P_4^r$.
\end{proof}

Each edge $e$ is identified with its incidence vector $x_e\in\{0,1\}^{V(H)}$.  Uniformity and linearity give
\begin{equation}\label{eq:gram}
 x_e^{\mathsf T}x_f=
 \begin{cases}
 r,&e=f,\\
 1,&ef\in E(L(H)),\\
 0,&ef\notin E(L(H)).
 \end{cases}
\end{equation}
Equivalently,
\[
 N(H)N(H)^{\mathsf T}=A(L(H))+rI.
\]
The argument will use the incidence vectors directly.

The following elementary observation will be used in the equality analysis.

\begin{lemma}[No universal transversal]\label{lem:no-transversal}
Let $\cS$ be an $S(2,r,r^2)$ contained in a linear hypergraph $H$, and let $f\in E(H)\setminus E(\cS)$.  Then $f$ cannot intersect every block of $\cS$.
\end{lemma}

\begin{proof}
The design has $r(r+1)$ blocks, and every point of the design lies in $r+1$ blocks.  If $f$ met every block, linearity would force every block to meet $f$ in exactly one point.  Counting incidences between the points of $f$ and the blocks of $\cS$ gives
\[
 r(r+1)
 =\sum_{v\in f\cap V(\cS)}(r+1)
 \le r(r+1).
\]
Thus every point of $f$ lies in $V(\cS)$.  Any two distinct points of $f$ lie together in a unique block of $\cS$, and that block would meet $f$ in at least two points, contradicting linearity.
\end{proof}

\section{The join-node rank decomposition}

Let $\cF$ be a nonempty edge family whose intersection graph $G$ is a connected cograph with at least two vertices.  For a subfamily $\cA\subseteq\cF$, write $N(\cA)$ for the matrix whose rows are the incidence vectors of the edges in $\cA$.  Let
\[
 \cF=\cF_1\mathbin{\dot\cup}\cdots\mathbin{\dot\cup}\cF_p,
 \qquad p\ge2,
\]
be the connected components of $\overline G$.  Thus every edge in $\cF_i$ intersects every edge in $\cF_j$ for $i\ne j$.  Write
\[
 m_i=|\cF_i|,
 \qquad R_i=\operatorname{span}\{x_e:e\in\cF_i\},
 \qquad \rho_i=\dim R_i,
\]
and define the row-difference space
\[
 D_i=\operatorname{span}\{x_e-x_f:e,f\in\cF_i\}.
\]
All vector spaces and ranks are over $\R$.

\begin{lemma}[Orthogonal join decomposition]\label{lem:join-rank}
The spaces $D_1,\ldots,D_p$ are mutually orthogonal, and
\[
 \dim D_i=\rho_i-1.
\]
Let $D=D_1\oplus\cdots\oplus D_p$.  For $e\in\cF_i$, let $c_i$ be the orthogonal projection of $x_e$ onto $D^\perp$.  This is independent of the choice of $e\in\cF_i$, and
\begin{equation}\label{eq:ci-products}
 \one^{\mathsf T}c_i=r,
 \qquad c_i^{\mathsf T}c_j=1\quad(i\ne j).
\end{equation}
If
\[
 z=\bigl|\{i:\|c_i\|^2=1\}\bigr|,
\]
then
\begin{equation}\label{eq:rank-decomposition}
 \rank N(\cF)
 =\sum_{i=1}^p\rho_i-\max\{z-1,0\}.
\end{equation}
Moreover, all vectors $c_i$ with $\|c_i\|^2=1$ are equal.
\end{lemma}

\begin{proof}
For $i\ne j$, every row from $\cF_i$ has inner product one with every row from $\cF_j$.  Hence
\[
 (x_e-x_f)^{\mathsf T}x_g=0
 \qquad(e,f\in\cF_i,\ g\in\cF_j),
\]
so $D_i\perp R_j$, and in particular the $D_i$ are mutually orthogonal.

Every vector in $D_i$ has coordinate sum zero.  Conversely, if
\[
 y=\sum_{e\in\cF_i}\alpha_e x_e\in R_i
 \quad\text{and}\quad
 \one^{\mathsf T}y=0,
\]
then $r\sum_e\alpha_e=0$.  After fixing $f\in\cF_i$,
\[
 y=\sum_{e\ne f}\alpha_e(x_e-x_f)\in D_i.
\]
Thus $D_i$ is the kernel in $R_i$ of the nonzero coordinate-sum functional, proving $\dim D_i=\rho_i-1$.

Rows in the same family $\cF_i$ differ by a vector in $D_i\subseteq D$, so they have the same projection $c_i$ onto $D^\perp$.  A row in $R_i$ is orthogonal to $D_j$ for $j\ne i$, so $x_e-c_i\in D_i$.  Since every vector in $D$ has coordinate sum zero, $\one^{\mathsf T}c_i=r$.  Orthogonality of the decomposition and the cross-intersection condition give $c_i^{\mathsf T}c_j=1$ for $i\ne j$.

The total row space is the orthogonal direct sum
\[
 D\oplus\operatorname{span}\{c_1,\ldots,c_p\}.
\]
It remains to determine the dimension of the second summand.  Suppose
\[
 \sum_{i=1}^p a_i c_i=0,
 \qquad S=\sum_{i=1}^p a_i.
\]
Taking coordinate sums gives $rS=0$, hence $S=0$.  Taking the inner product with $c_j$ and using \eqref{eq:ci-products} gives
\[
 0=a_j\|c_j\|^2+\sum_{i\ne j}a_i
   =a_j(\|c_j\|^2-1)+S.
\]
Therefore $a_j=0$ unless $\|c_j\|^2=1$.  If $\|c_i\|=\|c_j\|=1$, then $c_i^{\mathsf T}c_j=1$, so $\|c_i-c_j\|^2=0$ and $c_i=c_j$.  The relation space consequently has dimension $z-1$ when $z\ge1$ and dimension zero when $z=0$.  Combining this with $\dim D_i=\rho_i-1$ proves \eqref{eq:rank-decomposition}.
\end{proof}

A co-component $\cF_i$ with $\|c_i\|^2=1$ will be called \emph{balanced}.  The next lemma combines linearity with the connectedness of the complementary component.

\begin{lemma}[Balanced co-components]\label{lem:balanced}
Every balanced co-component $\cF_i$ has
\[
 \rho_i=m_i
 \qquad\text{and}\qquad
 m_i\ge r.
\]
If $m_i=r$, then the $r$ edges in $\cF_i$ are pairwise disjoint and
\begin{equation}\label{eq:c-average}
 c_i=\frac1r\sum_{e\in\cF_i}x_e.
\end{equation}
\end{lemma}

\begin{proof}
Fix a balanced component and abbreviate $c=c_i$, $m=m_i$.  Since $x_e-c\in D_i$ and $c\perp D_i$,
\[
 c^{\mathsf T}x_e=\|c\|^2=1
 \qquad(e\in\cF_i).
\]
Let $Q$ be the disjointness graph on $\cF_i$, which is the connected graph $\overline G[\cF_i]$.  Put $y_e=x_e-c$.  By \eqref{eq:gram},
\[
 y_e^{\mathsf T}y_f=
 \begin{cases}
 r-1,&e=f,\\
 0,&e\ne f\text{ and }e\cap f\ne\varnothing,\\
 -1,&e\cap f=\varnothing.
 \end{cases}
\]
Thus the Gram matrix of the vectors $y_e$ is
\begin{equation}\label{eq:centered-gram}
 K=(r-1)I-A(Q).
\end{equation}

The vector $c$ lies in the affine hull of the rows in $\cF_i$: indeed, $x_e-c\in D_i$, and $D_i$ is spanned by row differences.  Hence there are scalars $\lambda_e$ with $\sum_e\lambda_e=1$ and
\[
 c=\sum_e\lambda_e x_e.
\]
Equivalently, $\sum_e\lambda_e y_e=0$, so $K$ is singular.  Since $K$ is positive semidefinite, every eigenvalue of $A(Q)$ is at most $r-1$.  Singularity shows that $r-1$ is an eigenvalue; because $Q$ is connected, Perron--Frobenius implies that the kernel of $K$ is one-dimensional and is spanned by a vector with all entries positive.  In particular, the spectral radius of $Q$ is $r-1$, and therefore
\[
 r-1\le |V(Q)|-1=m-1,
\]
which gives $m\ge r$.

Suppose now that $\sum_e a_e x_e=0$.  Taking coordinate sums gives $\sum_ea_e=0$, so also $\sum_ea_e y_e=0$.  Therefore $a=(a_e)$ belongs to the kernel of $K$.  A nonzero vector in that kernel has all coordinates of one sign, whereas $\sum_ea_e=0$.  Thus $a=0$, proving that the rows are linearly independent and $\rho_i=m_i$.

If $m=r$, then the connected graph $Q$ has order $r$ and spectral radius $r-1$, so $Q=K_r$.  Hence the edges in $\cF_i$ are pairwise disjoint.  In this case \eqref{eq:centered-gram} is $rI-J$, whose kernel is spanned by the all-ones vector.  The affine coefficients above are therefore all $1/r$, which proves \eqref{eq:c-average}.
\end{proof}

Balanced co-components of minimum size are parallel classes.  Their number is bounded by an orthogonal-dimension argument.

\begin{lemma}[Parallel-class bound]\label{lem:parallel-bound}
At a fixed join node, at most $r+1$ balanced co-components can have size exactly $r$.
\end{lemma}

\begin{proof}
Let $\cF_1,\ldots,\cF_a$ be the balanced co-components of size $r$.  By Lemma~\ref{lem:join-rank}, their vectors $c_i$ are equal to a common vector $c$.  By Lemma~\ref{lem:balanced}, each $\cF_i$ consists of $r$ pairwise disjoint $r$-sets and
\[
 c=\frac1r\sum_{e\in\cF_i}x_e.
\]
Consequently every $\cF_i$ partitions the same set $U$ of $r^2$ points, and $c=r^{-1}\one_U$.  Edges from distinct classes meet in exactly one point.

For each $i$, let
\[
 W_i=\operatorname{span}\left\{x_e-\frac1r\one_U:e\in\cF_i\right\}.
\]
Since the $r$ edges in $\cF_i$ partition $U$, the space $W_i$ has dimension $r-1$ and lies in $\one_U^\perp$.  If $e\in\cF_i$ and $f\in\cF_j$ with $i\ne j$, then
\[
 \left(x_e-\frac1r\one_U\right)^{\mathsf T}
 \left(x_f-\frac1r\one_U\right)
 =1-1-1+1=0.
\]
Thus the spaces $W_1,\ldots,W_a$ are mutually orthogonal subspaces of the $(r^2-1)$-dimensional space $\one_U^\perp$.  Hence
\[
 a(r-1)\le r^2-1=(r-1)(r+1),
\]
and $a\le r+1$.
\end{proof}

\section{Proof of the rank theorem}

The proof is by induction on $m=|E(H)|$.  Vertices not contained in any edge may be deleted, so only the family of incidence row vectors is relevant.  Write
\[
 \rho(H)=\rank N(H),
 \qquad
 \sigma(H)=(r+1)\rho(H)-r|E(H)|.
\]

If $m=0$, the statement is immediate.  If $m=1$, then $\rho(H)=1$ and $\sigma(H)=1$, so the inequality is strict.  Suppose next that $L(H)$ is disconnected, with edge families $\cF^{(1)},\ldots,\cF^{(s)}$.  Distinct line-graph components have disjoint vertex supports, so the corresponding row spaces are supported on disjoint coordinate sets.  Therefore
\[
 \rho(H)=\sum_{j=1}^s\rho(\cF^{(j)}),
 \qquad
 \sigma(H)=\sum_{j=1}^s\sigma(\cF^{(j)}).
\]
The induction hypothesis proves $\sigma(H)\ge0$.

Now assume that $L(H)$ is connected and $m\ge2$.  Use the notation of the preceding section for the complement components $\cF_1,\ldots,\cF_p$, and let $B$ be the set of balanced indices, with $z=|B|$.

If $z\le1$, Lemma~\ref{lem:join-rank} and induction give
\[
 \rho(H)=\sum_i\rho_i
 \ge \frac{r}{r+1}\sum_i m_i
 =\frac{rm}{r+1}.
\]

Suppose $z\ge2$.  Balanced components have full row rank by Lemma~\ref{lem:balanced}.  Hence Lemma~\ref{lem:join-rank} gives
\begin{align*}
 \rho(H)
 &=\sum_{i\notin B}\rho_i+\sum_{i\in B}m_i-z+1\\
 &\ge \frac{r}{r+1}\sum_{i\notin B}m_i
      +\sum_{i\in B}m_i-z+1.
\end{align*}
It is enough to prove
\begin{equation}\label{eq:balanced-mass}
 \sum_{i\in B}m_i\ge(r+1)(z-1).
\end{equation}
Let $a$ be the number of balanced components of size $r$.  By Lemma~\ref{lem:balanced}, every other balanced component has size at least $r+1$, and by Lemma~\ref{lem:parallel-bound}, $a\le r+1$.  Therefore
\[
 \sum_{i\in B}m_i
 \ge ar+(z-a)(r+1)
 =(r+1)z-a
 \ge(r+1)(z-1),
\]
proving \eqref{eq:balanced-mass} and hence \eqref{eq:rank-bound}.

It remains to classify equality, again by induction on $m$.  If the line graph is disconnected, equality $\sigma(H)=0$ holds if and only if equality holds in every line-graph component.  The induction hypothesis then gives a vertex-disjoint union of $S(2,r,r^2)$ systems.

Assume that the line graph is connected.  If $z\le1$, the exact rank formula gives
\[
 \sigma(H)=\sum_{i=1}^p\sigma(\cF_i).
\]
Thus equality would force equality in every $\cF_i$.  By induction, every nonempty line-graph component inside each $\cF_i$ is an $S(2,r,r^2)$.  Choose one such design inside $\cF_1$ and any edge in $\cF_2$.  Because $\cF_1$ and $\cF_2$ are joined, that edge intersects every block of the design, contradicting Lemma~\ref{lem:no-transversal}.  Hence equality is impossible when $z\le1$.

Finally, suppose $z\ge2$.  The exact formula is
\begin{equation}\label{eq:sigma-exact}
 \sigma(H)
 =\sum_{i\notin B}\sigma(\cF_i)
  +\sum_{i\in B}m_i-(r+1)(z-1).
\end{equation}
The estimate used for \eqref{eq:balanced-mass} gives the stronger inequality
\[
 \sum_{i\in B}m_i-(r+1)(z-1)\ge r+1-a\ge0,
\]
where $a$ is the number of balanced components of size $r$.  If $\sigma(H)=0$, then $a=r+1$.  These $r+1$ minimal balanced components are $r+1$ parallel classes on a common set $U$ of $r^2$ points.  Every two blocks in different classes meet once, while blocks in the same class are disjoint.  The total number of point pairs contained in their blocks is
\[
 r(r+1)\binom r2=\binom{r^2}{2}.
\]
Linearity implies that no pair is counted twice, so every pair of points of $U$ lies in exactly one block.  The resulting family is an $S(2,r,r^2)$.

There can be no additional edge in any other complement component: such an edge would intersect every block of this $S(2,r,r^2)$, contradicting Lemma~\ref{lem:no-transversal}.  Hence the connected equality case is exactly one $S(2,r,r^2)$.  Conversely, the incidence matrix of an $S(2,r,r^2)$ has full column rank because
\[
 N^{\mathsf T}N=rI+J,
\]
and therefore its rank is $r^2$.  Thus equality holds in \eqref{eq:rank-bound}, proving Theorem~\ref{thm:rank}.

\section{Deduction of the Tur\'an bound}

Let $H$ be a linear $r$-uniform $P_4^r$-free hypergraph with $m$ edges and $n$ vertices.  By \cref{lem:line-cograph}, its line graph is a cograph.  The incidence-rank theorem gives
\[
 \frac{rm}{r+1}\le\rank N(H)\le n,
\]
which is equivalent to \eqref{eq:main-bound}.

If equality holds in \eqref{eq:main-bound}, then
\[
 n=\rank N(H)=\frac{rm}{r+1}.
\]
Thus equality also holds in \cref{thm:rank}, and the edge-containing components are vertex-disjoint copies of $S(2,r,r^2)$.  The equality $n=\rank N(H)$ rules out isolated vertices.  Conversely, every vertex-disjoint union of $S(2,r,r^2)$ systems is $P_4^r$-free and has edge density $(r+1)/r$.  This proves \cref{thm:main}.

\begin{corollary}\label{cor:extremal}
For every $r\ge2$ and every positive integer $n$,
\[
 \exlin_r(n,P_4^r)\le\frac{r+1}{r}n.
\]
Equality holds if and only if an $n$-point set can be partitioned into supports of Steiner systems $S(2,r,r^2)$.  Equivalently, equality holds precisely when $r^2\mid n$ and an $S(2,r,r^2)$ exists.
\end{corollary}


\begin{thebibliography}{99}

\bibitem{AdakVerma2026}
R.~Adak and P.~Verma,
\newblock Linear Tur\'an numbers of uniform hypertrees,
\newblock arXiv:2607.16854, 2026.

\bibitem{CollierCartainoGraberJiang2018}
C.~Collier-Cartaino, N.~Graber, and T.~Jiang,
\newblock Linear Tur\'an numbers of linear cycles and cycle-complete Ramsey numbers,
\newblock \emph{Combinatorics, Probability and Computing} 27 (2018), 358--386.

\bibitem{ErdosGallai1959}
P.~Erd\H{o}s and T.~Gallai,
\newblock On maximal paths and circuits of graphs,
\newblock \emph{Acta Mathematica Academiae Scientiarum Hungaricae} 10 (1959), 337--356.

\bibitem{Ghorbani2019}
E.~Ghorbani,
\newblock Cographs: eigenvalues and Dilworth number,
\newblock \emph{Discrete Mathematics} 342 (2019), 2797--2803.

\bibitem{GyarfasRuszinkoSarkozy2022}
A.~Gy\'arf\'as, M.~Ruszink\'o, and G.~N.~S\'ark\"ozy,
\newblock Linear Tur\'an numbers of acyclic triple systems,
\newblock \emph{European Journal of Combinatorics} 99 (2022), Article 103435.

\bibitem{JacobsTrevisanTura2018}
D.~P.~Jacobs, V.~Trevisan, and F.~C.~Tura,
\newblock Eigenvalue location in cographs,
\newblock \emph{Discrete Applied Mathematics} 245 (2018), 220--235.

\bibitem{Royle2003}
G.~F.~Royle,
\newblock The rank of a cograph,
\newblock \emph{Electronic Journal of Combinatorics} 10 (2003), Note N11.

\bibitem{Seinsche1974}
D.~Seinsche,
\newblock On a property of the class of $n$-colorable graphs,
\newblock \emph{Journal of Combinatorial Theory, Series B} 16 (1974), 191--193.

\bibitem{ZhangWang2026}
L.-P.~Zhang and L.-G.~Wang,
\newblock The linear Tur\'an numbers of acyclic linear $4$-graphs,
\newblock \emph{Acta Mathematicae Applicatae Sinica, English Series} 42 (2026), 829--841.

\end{thebibliography}
\end{document}